\documentclass[a4paper,11pt]{article}

\usepackage{amsmath,amssymb,amsthm,mathrsfs}
\usepackage[utf8]{inputenc}
\usepackage{a4}
\usepackage[english]{babel}
\usepackage{enumerate}
\usepackage{graphicx}
\usepackage{pinlabel}
\usepackage{tikz}
\usepackage{color}
\usepackage{float}
\usepackage[hidelinks]{hyperref}

\numberwithin{equation}{section}

\newtheorem{theorem}{Theorem}[section]
\newtheorem{proposition}[theorem]{Proposition}
\newtheorem{lemma}[theorem]{Lemma}
\newtheorem{question}[theorem]{Question}

\theoremstyle{definition}
\newtheorem{example}[theorem]{Example}
\newtheorem{definition}[theorem]{Definition}

\newtheoremstyle{customNumber}
     {}          
     {}          
     {\itshape}  
     {}          
     {\bfseries} 
     {.}         
     { }         
     {\thmname{#1}\thmnumber{ #2}\thmnote{ #3}}
\theoremstyle{customNumber}

\renewcommand{\O}{\mathcal O}

\renewcommand{\phi}{\varphi}
\renewcommand{\rho}{\varrho}

\DeclareMathOperator{\R}{\mathbb{R}}

\DeclareMathOperator{\sgn}{sgn}

\newcommand{\norm}[1]{\|#1\|}
\newcommand{\abs}[1]{\lvert#1\rvert}

\DeclareMathOperator\Id{Id}

\DeclareMathOperator\intt{int}

\title{Conical geodesic bicombings on subsets of normed vector spaces}

\author{Giuliano Basso and Benjamin Miesch}

\date{\today}

\begin{document}


\maketitle


\begin{abstract} 
We prove existence and uniqueness results for conical geodesic bicombings on subsets of normed vector spaces. Concerning existence,  we give a first example of a non-consistent convex geodesic bicombing.  Furthermore, we show that under a mild geometric assumption on the norm, a conical bicombing on an open subset of a normed space locally consists of linear geodesics. As an application, we obtain by the use of a Cartan-Hadamard type result that if a closed convex subset of a Banach space has non-empty interior, then it admits a unique consistent conical geodesic bicombing, namely the one given by linear segments.
\end{abstract}


\section{Introduction}
Given a metric space \(X=(X, d)\), we say that $\sigma\colon X\times X\times [0,1]\to X$ is a (geodesic) \textit{bicombing} if for all \(p\), \(q\in X\), the path $\sigma_{pq}(\cdot):=\sigma(p,q,\cdot)$ satisfies
\begin{equation*}
\sigma_{pq}(0)=p,\quad \sigma_{pq}(1)=q, \quad \textrm{and} \quad d(\sigma_{pq}(t), \sigma_{pq}(s))=\abs{t-s}\cdot d(p, q)
\end{equation*}
for all $s,t\in [0,1]$. Essentially, bicombings distinguish a class of geodesics in a metric space. For example, the collection of all oriented linear segments of a Banach space is a bicombing. In this article we mainly work with bicombings satisfying the following weak, but non-coarse, non-positive curvature condition: We say that $\sigma$ is \textit{conical} if 
\begin{equation}\label{Eq:conical}
d(\sigma_{pq}(t), \sigma_{p'q'}(t)) \leq (1-t)\cdot d(p,p')+t\cdot d(q,q')
\end{equation}
for all $p$, $q$, $p'$, $q'\in X$ and all $t\in [0,1]$. As it seems to be a common misconception, we already caution the reader here that (\ref{Eq:conical}) generally does \emph{not} imply that the function $t \mapsto d(\sigma_{pq}(t),\sigma_{p'q'}(t))$ is convex. This will become clear from some of the examples we construct in this article. 

The study of spaces with distinguished geodesics goes back to the influential work of Busemann \cite{busemann1987spaces}. Recently, the notion of a conical bicombing was coined by Lang \cite{lang2013injective} in connection with injective metric spaces (also called hyperconvex spaces by some authors), where conical bicombings are obtained naturally. Other examples of spaces with a conical bicombing are convex subsets of normed spaces and Busemann spaces. However, the class of metric spaces that admit conical bicombings is by no means limited to these examples. It easily follows that it is closed under ultralimits, \(\ell_\infty\)-products and also 1-Lipschitz retractions. 

In the past century, notions related to conical bicombings have been extensively considered in metric fixed point theory, most notable W-convexity mappings \cite{takahashi1970}, and hyperbolic spaces in the sense of Reich and Shafrir \cite{REICH1990537}. As a more recent trend, some authors became interested in transferring classical results from the theory of CAT(0) spaces to metric spaces with a conical bicombing \cite{basso2015fixed, descombes2015asymptotic, descombes2015flats, miesch2015cartan, kell2016sectional}. Most interestingly, it seems worth to point out that the study of conical bicombings may also lead to new results about word hyperbolic groups. Descombes and Lang \cite{lang1} proved that every Gromov hyperbolic group acts geometrically on a proper, finite dimensional metric space with a unique consistent conical bicombing (the definitions are recalled below). The question whether hyperbolic groups act geometrically on CAT(0) spaces is one of the most important open questions in geometric group theory. 

Our first result deals with convex bicombings. From now on, we abbreviate $D(X):=X\times X\times [0,1]$. We say that a bicombing $\sigma\colon D(X)\to X$ is \textit{convex} if the function $t\mapsto d(\sigma_{pq}(t), \sigma_{p'q'}(t))$ is convex on $[0,1]$ for all $p$, $q$, $p'$, $q'$ in $X$. Clearly, any convex bicombing is also conical. However, if the underlying metric space is not uniquely geodesic, then a conical  bicombing is not necessarily convex. Examples of non-convex conical  bicombings are ubiquitous; for instance, such bicombings may be obtained via 1-Lipschitz retractions of linear segments (see \cite[Example 2.2]{lang1} or the constructions in Section~\ref{Sec:reversible}). 

In \cite{lang1}, it is shown that metric spaces of finite combinatorial dimension in the sense of Dress \cite{dress} possess at most one convex  bicombing. If it exists, this unique convex  bicombing, say $\sigma\colon D(X)\to X$, is \textit{consistent}, that is, for all $p$, $q$ in \(X\) we have that $\sigma_{p'q'}([0,1])\subset \sigma_{pq}([0,1])$ whenever $p'=\sigma_{pq}(s)$ and $q'=\sigma_{pq}(t)$ with $s \leq t$. Clearly, every consistent conical  bicombing is convex. In Section \ref{Sec:convex_non-consistent}, we show that the converse does not hold by proving the following theorem.  

\begin{theorem}\label{Thm:convex_non-consistent}
There is a compact metric space that admits a convex  bicombing which is not consistent. 
\end{theorem}

Although there is a non-consistent convex  bicombing on the space considered in Section \ref{Sec:convex_non-consistent}, this space also admits a consistent convex  bicombing. We suspect that this is a general phenomenon.

\begin{question}\label{Question:properConsistent}
	Let $X$ be a proper metric space with a convex  bicombing. Does $X$ also admit a consistent convex  bicombing?
\end{question}

Recall that a metric space is called proper if every of its closed bounded subsets is compact. 

We now turn our attention to the natural notion of a reversible bicombing. A bicombing $\sigma$ is called \textit{reversible} if $\sigma_{pq}(t)=\sigma_{qp}(1-t)$ for all $p$, $q\in X$ and all $t\in [0,1]$. Notice that for reversible bicombings to establish the conical inequality \eqref{Eq:conical}, it suffices to show that
\[
d(\sigma_{vp}(t), \sigma_{vq}(t)) \leq t\cdot d(p, q)
\]
for all \(v\), \(p\), \(q\in X\) and all \(t\in [0,1]\). Thus, one might ask if every conical bicombing is automatically reversible. However, as it turns out this is not the case. A simple example of a non-reversible concial bicombing is constructed in Section~\ref{Sec:reversible}. Moreover, in Proposition \ref{Prop:convex'}, by modifying our non-consistent convex bicombing from Theorem~\ref{Thm:convex_non-consistent} slightly, we even obtain a first example of a non-reversible convex bicombing. 

In \cite{basso2015fixed}, a barycentric construction has been used to obtain fixed point theorems in metric spaces admitting conical  bicombings. This barycentric construction motivated the following definition. We say that $\sigma$ has the \textit{midpoint property} if $\sigma_{pq}(\frac{1}{2})=\sigma_{qp}(\frac{1}{2})$ for all $p,q$ in $X$. It seems natural to ask if at least every conical  bicombing that has the midpoint property is automatically reversible. However, in Lemma~\ref{lem:midpoint-property}, we show that this is also not the case, by constructing a non-reversible conical bicombing which has the midpoint property. On the positive side, we finish Section~\ref{Sec:reversible} with the following proposition. 

\begin{proposition}\label{Prop:reversible}
	Let $X$ be a complete metric space with a conical  bicombing. Then $X$ also admits a reversible conical  bicombing.
\end{proposition}

This generalizes the result for proper metric spaces established in \cite{descombes2015asymptotic}. Next, we consider conical bicombings on convex subsets of normed spaces. Quite surprisingly, it follows directly from a result of Gähler and Murphy \cite[Theorem 1]{murphy} that the only conical bicombing on a normed vector space is the one consisting of linear segments. With a mild geometric assumption on the norm, we obtain in Section~\ref{Sec:uniqueness} a local version of this uniqueness result. 

\begin{theorem}\label{Thm:linearInterior}
Let $V$ be a normed vector space such that its unit ball is the closed convex hull of its extreme points. Let $A\subset V$  and \(\sigma\) a conical bicombing on \(A\). If \(p_0\in A\) and $r\geq 0$ are such that $B_{2r}(p_0)\subset A$, then on \(B_r(p_0)\) the bicombing \(\sigma\) is given by linear segments, that is, $\sigma(p,q,t)=(1-t)p+tq$ for all $p,q\in B_{r}(p_0)$ and all $t\in [0,1]$. 
\end{theorem}

We do not know if Theorem~\ref{Thm:linearInterior} remains true if we drop the assumption that the unit ball of \(V\) is the closed convex hull of its extreme points. But how common is this property? By an application of the theorems of Banach-Alaoğlu and Kreĭn-Mil’man, it follows that the unit ball of a dual Banach space is the closed convex hull of its extreme points. Moreover, using a classification result, due to Nachbin, Goodner, and Kelley, see \cite{kelley1952banach}, and a result of Goodner, see \cite[Theorem 6.4]{goodner1950projections}, it is readily verified that Theorem~\ref{Thm:linearInterior} also holds for every injective Banach space. We remark in passing that the classical Mazur-Ulam theorem is a direct consequence of Theorem~\ref{Thm:linearInterior}, as every isometry between two normed spaces extends to an isometry between their linear injective hulls, which are injective Banach spaces and by the above satisfy the assumptions of Theorem~\ref{Thm:linearInterior}.

In \cite{miesch2015cartan}, the second named author generalized the classical Cartan-Hadamard theorem to metric spaces that locally admit a consistent convex  bicombing. With Theorem~\ref{Thm:linearInterior} at hand, it is possible to use this generalized Cartan-Hadamard theorem to obtain the following uniqueness result.

\begin{theorem}\label{Thm:uniqueness}
Let $E$ be a Banach space such that its unit ball is the closed convex hull of its extreme points. Suppose that $C\subset E$ is a closed convex subset of $E$ with non-empty interior. If $\sigma$ is a consistent conical  bicombing on \(C\), then $\sigma(p,q,t)=(1-t)p+tq$ for all $p,q\in C$ and all $t\in[0,1]$. 
\end{theorem}

The proof of Theorem~\ref{Thm:uniqueness} is carried out in Section \ref{sec:ProofofTheorem1.4}. In Example~\ref{Ex:twoDistinct} we use a non-affine isometry first introduced by Schechtman to construct two distinct consistent conical  bicombings on a closed convex subset \(B\subset L^1([0,1])\) with empty interior. As it is possible to consider \(B\) as a subset of the injective hull of \(L^1([0,1])\), it follows that the assumption in Theorem~\ref{Thm:uniqueness} of \(C\) having non-empty interior is necessary. 

Due to Theorems \ref{Thm:linearInterior} and \ref{Thm:uniqueness} it appears that the geometry of a convex subset \(C\) with non-empty interior is very restricted in the sense that it is difficult to construct a conical  bicombing on \(C\) that is not given by linear segments. In this perspective, we suspect that a negative answer to the following question would result in an interesting geometric construction. 

\begin{question}
Let \(C\subset E\) be a convex subset of a Banach space \(E\). Suppose that \(C\) has non-empty interior. Is it true that \(C\) admits only one conical
 bicombing?
\end{question}

\subsection{Acknowledgements} We would like to thank Urs Lang for introducing us to conical geodesic bicombings and for his helpful remarks and guidance. We are also thankful for helpful suggestions of the anonymous referee. The authors gratefully acknowledge support from the Swiss National Science Foundation.


\section{A non-consistent convex  bicombing}\label{Sec:convex_non-consistent}


The aim of this section is to construct a convex  bicombing that is not consistent and therefore establish Theorem~\ref{Thm:convex_non-consistent}. To this end, we consider the following norm on $\mathbb{R}^2$:
\[
	\| (x,y) \| := \max \big\{ |x|,\, \tfrac{\sqrt{2}}{2}\cdot \| (x,y) \|_2 \big\},
\]
	where $\norm{\,\cdot\,}_2$ is the standard Euclidean norm. This should be thought of as an interpolation between a strictly convex norm and a polyhedral norm. Observe that $\| (x,y) \| = |x|$ if and only if $|y| \leq |x|$. In the following, we consider
\[
	X := \left\{ (x,y) \in \mathbb{R}^2 : -3 \leq x \leq 3, \, 0 \leq y \leq \tfrac{1}{32} \max \{0,1-x^2\} \right\}
\]
equipped with the metric $d$ induced by $\norm{\,\cdot\, }$, see Figure~\ref{Fig:convex_non-consistent}. The space $X$ naturally splits into three parts, namely $X=X_- \cup X_0 \cup X_+$ where
\[
X_0 :=\left\{ (x,y) \in \mathbb{R}^2 : -1 < x < 1, 0 \leq y \leq \tfrac{1}{32} (1-x^2) \right\},
\]
\(X_- := [-3,-1] \times \{0\}\) and	\(X_+ := [1,3] \times \{0\}\). 
We consider the following family of bicombings on \(X\).

\begin{figure}[t]
\labellist
\small\hair 2pt
\pinlabel $p$ [b] at 21 15
\pinlabel $q$ [b] at 310 15
\pinlabel {$\sigma_{pq}^\delta$} [bl] at 225 25
\endlabellist
\centering
\includegraphics[scale=0.9]{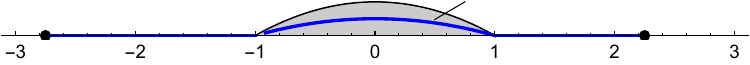}
\caption{The metric space $X$ with a geodesic $\sigma_{pq}^\delta$.}
\label{Fig:convex_non-consistent}
\end{figure}
	
\begin{definition}	
	For $0\leq \delta \leq \frac{1}{64}$ we define $\sigma^\delta \colon D(X) \to X$ as follows. We take $\sigma_{pq}^\delta$ to be the geodesic from $p$ to $q$ which is linear inside $X_0$, but if both endpoints lie on the antennas $X_-$, $X_+$ then we slightly modify it, see Figure~\ref{Fig:convex_non-consistent}. More precisely, $\sigma^\delta$ is defined as follows: For $p:=(p_x,p_y)$, $q:=(q_x,q_y) \in X$ with $p_x \leq q_x$ let \(\sigma_{pq}^\delta(t) := \left( x_{pq}(t),y_{pq}(t) \right)\) and \(\sigma_{qp}^\delta(t):= \sigma_{pq}^\delta(1-t)\), with 
	\[
			x_{pq}(t) := p_x + t(q_x-p_x), 
	\]
			
	and 
	\begin{align*}
		y_{pq}(t) &:= \begin{cases}		
		\delta \max\{q_x-p_x-4,0\} \max \{0,(1 - x_{pq}(t)^2) \} &\text{ for } p \in X_-,\, q \in X_+, \\[0.3em]
		\max \{0 , \frac{q_y}{q_x+1}(x_{pq}(t) + 1)\} &\text{ for } p \in X_-,\, q \in X_0, \\[0.3em]
		\max \{0 , \frac{p_y}{p_x-1}(x_{pq}(t) - 1)\} &\text{ for } p \in X_0,\, q \in X_+, \\[0.3em]
		p_y + t(q_y-p_y) &\text{ for } p,\, q \in X_0,\\[0.3em]
		0 &\text{ otherwise. }
		\end{cases}
	\end{align*}
\end{definition}

Observe that for $\delta=0$, $\sigma^\delta$ coincides with the piecewise linear bicombing, which is the unique consistent conical bicombing on $X$ by Theorem~\ref{Thm:uniqueness}. Hence, we have constructed a family of non-consistent convex  bicombings $\sigma^\delta$ converging to the unique consistent convex  bicombing $\sigma^0$.

\begin{proposition}\label{Prop:convex}
For any $\delta \in (0,\frac{1}{64}]$, the bicombing $\sigma^\delta$ is reversible and convex but not consistent.
\end{proposition}

Alternatively, we can also modify the geodesics leading from $X_+$ to $X_-$ so that we lose the reversibility. 
	
\begin{definition}
Define $\tilde{\sigma}^\delta \colon D(X) \to X$ by \(\tilde{\sigma}_{pq}^\delta(t) = \sigma_{pq}^\delta(t)\), except for $p \in X_+$ and $q \in X_-$ we let \(\tilde{\sigma}_{pq}^\delta(t) = (x_{pq}(t),0)\).
\end{definition}

As it turns out this new bicombing is still convex. 

\begin{proposition}\label{Prop:convex'}
For any $\delta \in (0,\frac{1}{64}]$, the map $\tilde{\sigma}^\delta$ is a convex  bicombing which is neither reversible nor consistent.
\end{proposition}
The quite technical proofs of Propositions~\ref{Prop:convex} and \ref{Prop:convex'} are given in the appendix. 
Let us proceed by the simple proof showing that we actually have defined  bicombings.
	
\begin{lemma}\label{Lem:bicombing}
For any $\delta \in [0,\frac{1}{64}]$, the maps $\sigma^\delta$ and $\tilde{\sigma}^\delta$ are bicombings.
\end{lemma}

\begin{proof}
	The linear case is clear. For the piecewise linear case observe that if $p \in X_-$, $q\in X_0$ (and similarly in all other cases) we have that the slope $m$ of $\sigma_{pq}^\delta$ satisfies
\[
		m = \frac{q_y}{q_x+1} \leq \frac{\frac{1}{32}(1-q_x^2)}{1+q_x} \leq  1
\]
and therefore
	$$
	d(\sigma_{pq}^\delta(s),\sigma_{pq}^\delta(t)) = |x_{pq}(s)-x_{pq}(t)| = |s-t|\cdot|q_x-p_x| = |s-t|d(p,q).
	$$
Finally, let $p \in X_-$, $q \in X_+$. For $x,x' \in [-1,1]$ we have 
	\begin{align*}
		| \delta(q_x-p_x-4) (1 - x^2) - \delta(q_x-p_x-4) (1 - x'^2) | \\
		\leq \delta\cdot |q_x-p_x-4| \cdot |x+x'|\cdot|x-x'| &\leq \tfrac{1}{16} |x-x'|
	\end{align*}
	and hence $d(\sigma_{pq}^\delta(s),\sigma_{pq}^\delta(t)) = |x_{pq}(s)-x_{pq}(t)|$.
\end{proof}

It is immediate that both families consist only of non-consistent bicombings. Furthermore, $\sigma^\delta$ is reversible and $\tilde{\sigma}^\delta$ is not. Hence, it remains to prove convexity in both cases. For given $p$, $q$, $p'$, $q' \in X$ we need to show that $f(t):= d(\sigma_{pq}^\delta(t),\sigma_{p'q'}^\delta(t))$ is convex on \([0,1]\). To this end, we use the following well-known characterization of convexity; see e.g. \cite[Lemma~3.5]{li2010some}.
	
\begin{lemma}\label{Lem:locally convex}
	A continuous function $f\colon [0,1] \to \mathbb{R}$ is convex if and only if for every $t \in (0,1)$ there is some $\tau_0 > 0$ such that  
	\begin{align*}
		2f(t) \leq f(t-\tau) + f(t+\tau)
	\end{align*}
	for all $\tau \in [0,\tau_0]$.
\end{lemma}
Notice that if
$d(\sigma_{pq}^\delta(t),\sigma_{p'q'}^\delta(t)) = |x_{pq}(t)-x_{p'q'}(t)|$, then 
\begin{align*}
	2d(\sigma_{pq}^\delta(t),\sigma_{p'q'}^\delta(t)) &= 2|x_{pq}(t)-x_{p'q'}(t)| \\
	&\leq |x_{pq}(t-\tau)-x_{p'q'}(t-\tau)|+|x_{pq}(t+\tau)-x_{p'q'}(t+\tau)| \\
	&\leq d(\sigma_{pq}^\delta(t-\tau),\sigma_{p'q'}^\delta(t-\tau))+d(\sigma_{pq}^\delta(t+\tau),\sigma_{p'q'}^\delta(t+\tau)),
\end{align*}
as  $t \mapsto |x_{pq}(t)-x_{p'q'}(t)|$ is convex on \([0,1]\).
Hence, in view of Lemma~\ref{Lem:locally convex}, to show convexity it remains to check that
\begin{equation}\label{Eq:ConvexityCondition}
\begin{split}
	&2 \|\sigma_{pq}^\delta(t)-\sigma_{p'q'}^\delta(t)\|_2 \\
	&\leq \|\sigma_{pq}^\delta(t-\tau)-\sigma_{p'q'}^\delta(t-\tau)\|_2+\|\sigma_{pq}^\delta(t+\tau)-\sigma_{p'q'}^\delta(t+\tau)\|_2
\end{split}
\end{equation}
for $\tau > 0$ sufficiently small, whenever $d(\sigma_{pq}^\delta(t),\sigma_{p'q'}^\delta(t)) = \tfrac{\sqrt{2}}{2} \|\sigma_{pq}^\delta(t)-\sigma_{p'q'}^\delta(t)\|_2$.

\begin{figure}[b]
\labellist
\small\hair 2pt
\pinlabel $p$ [b] at 9 15
\pinlabel $p'$ [b] at 68 15
\pinlabel $q'$ [b] at 296 15
\pinlabel $q$ [b] at 353 15
\endlabellist
\centering
\includegraphics[scale=0.9]{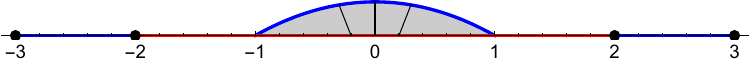}
\caption{The function $t \mapsto d(\sigma_{pq}^\delta(t),\sigma_{p'q'}^\delta(t))$ is convex.}
\label{Fig:convexity}
\end{figure}
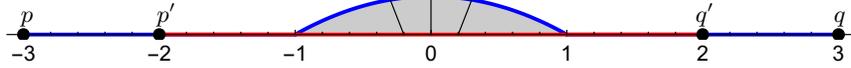

In this case, the main reason for convexity is that the modification in the $y$-direction is controlled by the speed difference in the $x$-direction. To illustrate this, let us consider $\sigma_{pq}^\delta$ and $\sigma_{p'q'}^\delta$ for \(p=(-3,0),\,\, q=(3,0),\,\, p'=(-2,0),\) and  \(q'=(2,0)\). Note that for $t \in [\frac{1}{3},\frac{2}{3}]$, $\sigma_{pq}^\delta(t)$ lies on the (concave) parabola $2\delta (1-x^2)$ while $\sigma_{p'q'}^\delta$ describes a linear segment on the $x$-axis.
	However, e.g. for $t= \frac{1}{2}$, we have
	\begin{align*}
		\|\sigma_{pq}(\tfrac{1}{2}\pm \tau)-\sigma_{p'q'}(\tfrac{1}{2}\pm\tau)\|_2^2 &= (3\tau-2\tau)^2 + 4 \delta^2 (1-9 \tau^2)^2 \\
		&= 4 \delta^2 + (1-72 \delta^2)\tau^2+324 \delta^2\tau^4 \\
		&\geq 4 \delta^2 = \|(\sigma_{pq}(\tfrac{1}{2})-\sigma_{p'q'}(\tfrac{1}{2})\|_2^2,
	\end{align*}
and consequently, \eqref{Eq:ConvexityCondition} follows. A similar calculation can also be carried out for all other pairings of geodesics of the bicombing. To this end, we shall distinguish several cases. This is done in the appendix, where detailed proofs of Propositions~\ref{Prop:convex} and \ref{Prop:convex'} are given.


\section{Reversibility of conical  bicombings}\label{Sec:reversible}


In this section, we present a simple example of a non-reversible conical  bicombing.  Then we modify this bicombing to satisfy the midpoint property while still being non-reversible. We finish the section with the proof of Proposition~\ref{Prop:reversible}.

To construct our non-reversible bicombing we need to introduce quite a bit of notation. In the following, we consider $\R^2$ equipped with the maximum norm \(\norm{\,\cdot\,}_\infty\) and we frequently use the reflection $s\colon \R^2\to \R^2$ about the \(x\)-axis defined by $(x,y)\mapsto (x,-y)$. Further, we set 
\begin{align*}
X &:= \big\{ (x,y)\in \R^2 : x\in[-2,1] \textrm{ and } \abs{x}-1\leq y \leq \abs{\abs{x}-1}\big\}, \\
A &:= \big\{ (x,y)\in \R^2 : \abs{x+1} \leq y \leq 1 \big\},
\end{align*}
and $X' := s(X)$, $A' := s(A)$. Notice that \(A\) and \(A'\) are filled triangles. Moreover, the union $X \cup X'$ is equal to a filled square with two antennas attached, see Figure~\ref{Fig:non-reversible}.
By construction, $f\colon X' \to X$ defined by
\begin{align*}
(x,y) \mapsto 
\begin{cases}
(x,y),  &\textrm{if } x \in [-1,1], \\
s(x,y), &\textrm{if } x \in [-2,-1]
\end{cases}
\end{align*}
is an isometry and the map $\bar{f} \colon X \cup X' \to X$ which by definition is equal to $\Id_{X}$ on \(X\) and equal to $f$ on \(X'\), is clearly 1-Lipschitz. We set \(Y:= X\cup A\) and \(Y':=X'\cup A'\), and we define the vertical projection $\pi\colon Y\cup Y'\to X \cup X'$ through the assignment
\begin{equation*}
(x,y)\mapsto \left(x, \sgn(y)\min\big\{ \abs{y}, \abs{\abs{x}-1}\big\}\right).
\end{equation*}
Observe that $\pi$ is a $1$-Lipschitz retraction that maps $Y$ to $X$ and \(Y'\) to \(X'\). Now, we are in a position to define our non-reversible conical bicombing.

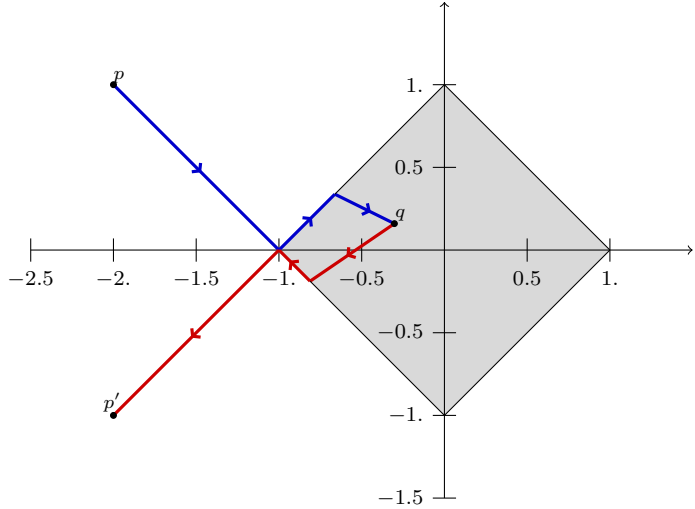
\begin{figure}[h]
\begin{center}
\definecolor{ccqqqq}{rgb}{0.8,0.,0.}
\definecolor{qqqqcc}{rgb}{0.,0.,0.8}
\begin{tikzpicture}[scale=2.2][line cap=round,line join=round,>=triangle 45,x=1.0cm,y=1.0cm]
\draw[->,color=black] (-2.5,0.) -- (1.5,0.);
\foreach \x in {-2.5,-2.,-1.5,-1.,-0.5,0.5,1.}
\draw[shift={(\x,0)},color=black] (0pt,2pt) -- (0pt,-2pt) node[below] {\footnotesize $\x$};
\draw[->,color=black] (0.,-1.5) -- (0.,1.5);
\foreach \y in {-1.5,-1.,-0.5,0.5,1.}
\draw[shift={(0,\y)},color=black] (2pt,0pt) -- (-2pt,0pt) node[left] {\footnotesize $\y$};
\clip(-2.5,-1.5) rectangle (1.5,1.5);
\fill[line width=0.pt,fill=black,fill opacity=0.15] (0.,1.) -- (1.,0.) -- (0.,-1.) -- (-1.,0.) -- cycle;
\draw [line width=1.2pt,color=qqqqcc] (-0.6621159940084833,0.33788400599151674)-- (-0.3023175665969685,0.1598205174219592);
\draw [line width=1.2pt,color=qqqqcc] (-0.45524417198530953,0.23550357431148639) -- (-0.49890263954679004,0.21513650130996748);
\draw [line width=1.2pt,color=qqqqcc] (-0.45524417198530953,0.23550357431148639) -- (-0.46553092105866156,0.28256802210350845);
\draw [line width=1.2pt,color=qqqqcc] (-1.,0.)-- (-0.6621159940084833,0.33788400599151674);
\draw [line width=1.2pt,color=qqqqcc] (-0.8097776180764485,0.19022238192355115) -- (-0.8044575233445005,0.1423415293360169);
\draw [line width=1.2pt,color=qqqqcc] (-0.8097776180764485,0.19022238192355115) -- (-0.8576584706639826,0.19554247665549945);
\draw [line width=1.2pt,color=qqqqcc] (-2.,1.)-- (-1.,0.);
\draw [line width=1.2pt,color=qqqqcc] (-1.4787196210722073,0.4787196210722073) -- (-1.5266004736597414,0.47339952634025895);
\draw [line width=1.2pt,color=qqqqcc] (-1.4787196210722073,0.4787196210722073) -- (-1.473399526340259,0.5266004736597415);
\draw [line width=1.2pt,color=ccqqqq] (-0.8117720522997935,-0.18822794770020657)-- (-1.,0.);
\draw [line width=1.2pt,color=ccqqqq] (-0.9271664050776903,-0.07283359492231044) -- (-0.8792855524901562,-0.06751350019036216);
\draw [line width=1.2pt,color=ccqqqq] (-0.9271664050776903,-0.07283359492231044) -- (-0.9324864998096383,-0.12071444750984471);
\draw [line width=1.2pt,color=ccqqqq] (-0.3023175665969685,0.1598205174219592)-- (-0.8117720522997935,-0.18822794770020657);
\draw [line width=1.2pt,color=ccqqqq] (-0.5818943748827438,-0.03118040952714361) -- (-0.578265677433406,0.01685824165382985);
\draw [line width=1.2pt,color=ccqqqq] (-0.5818943748827438,-0.03118040952714361) -- (-0.5358239414633559,-0.0452656719320771);
\draw [line width=0.4pt] (0.,1.)-- (1.,0.);
\draw [line width=0.4pt] (1.,0.)-- (0.,-1.);
\draw [line width=0.4pt] (-0.8117720522997935,-0.18822794770020662)-- (0.,-1.);
\draw [line width=0.4pt] (-0.6621159940084833,0.33788400599151674)-- (0.,1.);
\draw [line width=1.2pt,color=ccqqqq] (-1.,0.)-- (-2.,-1.);
\draw [line width=1.2pt,color=ccqqqq] (-1.5212803789277929,-0.5212803789277932) -- (-1.5266004736597412,-0.47339952634025895);
\draw [line width=1.2pt,color=ccqqqq] (-1.5212803789277929,-0.5212803789277932) -- (-1.4733995263402588,-0.5266004736597418);
\begin{scriptsize}
\draw [fill=black] (-2.,1.) circle (0.5pt);
\draw[color=black] (-1.9664629397455622,1.0486306344960463) node {$p$};
\draw [fill=black] (-0.3023175665969685,0.1598205174219592) circle (0.5pt);
\draw[color=black] (-0.2660954119405932,0.21098645413194422) node {$q$};
\draw [fill=black] (-2.,-1.) circle (0.5pt);
\draw[color=black] (-2.0065896070683933,-0.9175760643226843) node {$p'$};
\end{scriptsize}
\end{tikzpicture}
\caption{The blue line corresponds to $\sigma_{pq}$ and the red line corresponds to the image of $\sigma_{qp}$ under the isometry $f^{-1}$.}
\label{Fig:non-reversible}
\end{center}
\end{figure}

\begin{lemma}\label{Lem:non-reversible} 
Let $\lambda$ be the bicombing on \(\R^2\) given by linear segments. Then $\sigma\colon D(X) \to X$ defined by 
\begin{equation*}\label{eq:defbicombing}
(p,q,t)\mapsto 
\begin{cases}
(\pi \circ \lambda)\,(p, q, t), &\textrm{ if } p_x \leq q_x, \\
(f \circ \pi \circ \lambda)\,(f^{-1}(p),f^{-1}(q),t), &\textrm{ if } q_x \leq p_x . 
\end{cases}
\end{equation*}
is a non-reversible conical  bicombing on $X$ equipped with $\norm{\,\cdot\,}_{\infty}$. 
\end{lemma}

\begin{proof}
Both maps \(\sigma^{(1)} := \pi \circ \lambda\)  and  \(\sigma^{(2)} :=f \circ \pi \circ \lambda \circ (f^{-1}\times f^{-1} \times \Id_{[0,1]})\)
define bicombings on $X$. Thus, it follows that $\sigma\colon D(X)\to X$ is a bicombing as well. In the following we show that $\sigma$ is conical. As both maps $\sigma^{(1)}$ and $\sigma^{(2)}$ are conical bicombings with $\sigma^{(1)}_{pq} = \sigma^{(2)}_{pq}$ if $p_x$, $q_x \leq -1$ or $p_x$, $q_x \geq -1$, it remains to check the conical inequality~\eqref{Eq:conical} when 
\[
p_x, q'_x \leq -1 \,\text{ and }\,q_x, p'_x \geq -1 \quad \text{ or } \quad p'_x, q_x \leq -1 \,\text{ and }\, q'_x, p_x \geq -1.
\]
Suppose  $p_x$, $q'_x \leq -1$ and $q_x$, $p'_x \geq -1$. Since $\bar{f} \circ \pi $ is $1$-Lipschitz, we compute
\begin{align*}
\norm{\sigma_{pq}(t)-\sigma_{p^\prime q^\prime}(t)}_{\infty} &= \norm{(\bar{f} \circ \pi \circ \lambda)\,(p,q,t)-(\bar{f} \circ \pi \circ \lambda)\,(f^{-1}(p'),f^{-1}(q'),t)}_{\infty} \\[0.2em]
&\leq (1-t)\norm{p-f^{-1}(p')}_{\infty}+t\norm{q-f^{-1}(q')}_\infty
\end{align*}
for all $t \in [0,1]$. By our assumptions on the points \(p\), \(q\), \(p'\), \(q'\),  
\begin{align*}
\norm{p-f^{-1}(p')}_{\infty} &= \norm{p-p'}_\infty, \\ 
\norm{q-f^{-1}(q')}_{\infty} &= \norm{f^{-1}(q)-f^{-1}(q')}_{\infty}=\norm{q-q^\prime}_{\infty}. 
\end{align*}
Hence, \eqref{Eq:conical} follows. The other case is treated analogously. Thus, $\sigma$ is a conical bicombing on $X$. By construction, $\sigma$ is non-reversible; see Figure~\ref{Fig:non-reversible}.
\end{proof}

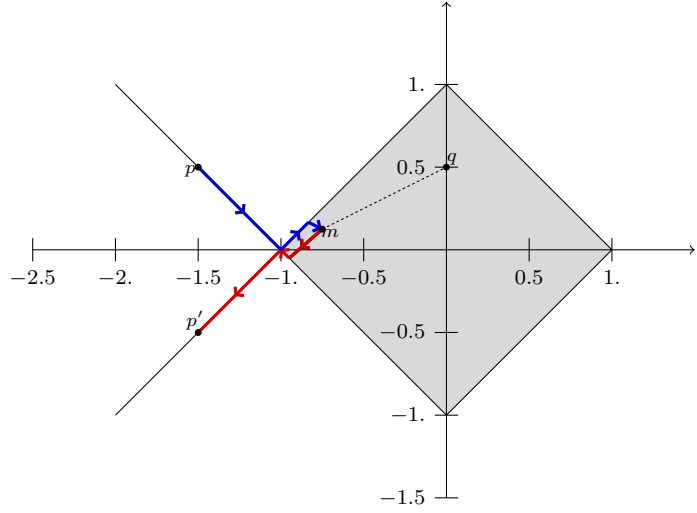
\begin{figure}
\begin{center}
\definecolor{ccqqqq}{rgb}{0.8,0.,0.}
\definecolor{qqqqcc}{rgb}{0.,0.,0.8}
\begin{tikzpicture}[scale=2.2][line cap=round,line join=round,>=triangle 45,x=1.0cm,y=1.0cm]
\draw[->,color=black] (-2.5,0.) -- (1.5,0.);
\foreach \x in {-2.5,-2.,-1.5,-1.,-0.5,0.5,1.}
\draw[shift={(\x,0)},color=black] (0pt,2pt) -- (0pt,-2pt) node[below] {\footnotesize $\x$};
\draw[->,color=black] (0.,-1.5) -- (0.,1.5);
\foreach \y in {-1.5,-1.,-0.5,0.5,1.}
\draw[shift={(0,\y)},color=black] (2pt,0pt) -- (-2pt,0pt) node[left] {\footnotesize $\y$};
\clip(-2.5,-1.5) rectangle (1.5,1.5);
\fill[line width=0.pt,fill=black,fill opacity=0.15] (0.,1.) -- (-1.,0.) -- (0.,-1.) -- (1.,0.) -- cycle;
\draw [line width=0.4pt] (-2.,1.)-- (-1.,0.);
\draw [line width=0.4pt] (0.,1.)-- (1.,0.);
\draw [line width=0.4pt] (1.,0.)-- (0.,-1.);
\draw [line width=0.4pt] (-1.,0.)-- (-2.,-1.);
\draw [line width=0.4pt] (-1.,0.)-- (0.,-1.);
\draw [line width=0.4pt] (0.,1.)-- (-1.,0.);
\draw [line width=0.4pt,dash pattern=on 1pt off 1pt] (-0.75,0.125)-- (0.,0.5);
\draw (-0.75,0.125)-- (-0.8333333333333334,0.16666666666666666);
\draw (-0.95,-0.05)-- (-1.,0.);
\draw [line width=1.2pt,color=ccqqqq] (-1.,0.)-- (-1.5,-0.5);
\draw [line width=1.2pt,color=ccqqqq] (-1.2714590229127505,-0.2714590229127505) -- (-1.2768237786409382,-0.22317622135906137);
\draw [line width=1.2pt,color=ccqqqq] (-1.2714590229127505,-0.2714590229127505) -- (-1.2231762213590616,-0.2768237786409382);
\draw [line width=1.2pt,color=ccqqqq] (-0.95,-0.05)-- (-1.,0.);
\draw [line width=1.2pt,color=ccqqqq] (-0.9964590229127503,-0.003540977087249497) -- (-0.9481762213590613,0.0018237786409382132);
\draw [line width=1.2pt,color=ccqqqq] (-0.9964590229127503,-0.003540977087249497) -- (-1.001823778640938,-0.0518237786409386);
\draw [line width=1.2pt] (-0.75,0.125)-- (-0.95,-0.05);
\draw [line width=1.2pt] (-0.8728389275354104,0.0175159384065156) -- (-0.8749800769918551,0.0660486594192633);
\draw [line width=1.2pt] (-0.8728389275354104,0.0175159384065156) -- (-0.8250199230081442,0.00895134058073656);
\draw [line width=1.2pt,color=ccqqqq] (-0.75,0.125)-- (-0.95,-0.05);
\draw [line width=1.2pt,color=ccqqqq] (-0.8728389275354104,0.0175159384065156) -- (-0.8749800769918551,0.0660486594192633);
\draw [line width=1.2pt,color=ccqqqq] (-0.8728389275354104,0.0175159384065156) -- (-0.8250199230081442,0.00895134058073656);
\draw [line width=1.2pt,color=qqqqcc] (-1.5,0.5)-- (-1.,0.);
\draw [line width=1.2pt,color=qqqqcc] (-1.2285409770872489,0.22854097708724938) -- (-1.276823778640938,0.22317622135906168);
\draw [line width=1.2pt,color=qqqqcc] (-1.2285409770872489,0.22854097708724938) -- (-1.2231762213590613,0.2768237786409385);
\draw [line width=1.2pt,color=qqqqcc] (-1.,0.)-- (-0.8333333333333334,0.16666666666666666);
\draw [line width=1.2pt,color=qqqqcc] (-0.8952076437539161,0.10479235624608386) -- (-0.8898428880257284,0.056509554692394756);
\draw [line width=1.2pt,color=qqqqcc] (-0.8952076437539161,0.10479235624608386) -- (-0.9434904453076052,0.11015711197427157);
\draw [line width=1.2pt,color=qqqqcc] (-0.8333333333333334,0.16666666666666666)-- (-0.75,0.125);
\draw [line width=1.2pt,color=qqqqcc] (-0.7645229111602532,0.13226145558012667) -- (-0.8086315138581743,0.11190363895031687);
\draw [line width=1.2pt,color=qqqqcc] (-0.7645229111602532,0.13226145558012667) -- (-0.7747018194751583,0.17976302771634972);
\begin{scriptsize}
\draw [fill=black] (-1.5,0.5) circle (0.5pt);
\draw[color=black] (-1.548413944424672,0.4786231079085647) node {$p$};
\draw [fill=black] (0.,0.5) circle (0.5pt);
\draw[color=black] (0.0347213401830625,0.5494342707983676) node {$q$};
\draw [fill=black] (-1.5,-0.5) circle (0.5pt);
\draw[color=black] (-1.5180663031861852,-0.4267481890396306) node {$p'$};
\draw [fill=black] (-0.75,0.125) circle (0.5pt);
\draw[color=black] (-0.7037379299534526,0.0992775924274773) node {$m$};
\end{scriptsize}
\end{tikzpicture}
\caption{The blue line corresponds to \(\tau_{pq}|_{[0,\frac{1}{2}]}\) and the red line corresponds to the image of \(\tau_{qp}|_{[\frac{1}{2},1]}\) under the isometry \(f^{-1}\). The point \(m\) is equal to \(\frac{1}{2}\left(\sigma_{pq}(\frac{1}{2})+\sigma_{qp}(\frac{1}{2})\right)\).}
\label{Fig:midpoint}
\end{center}
\end{figure}

Next, we use the conical bicombing from Lemma~\ref{Lem:non-reversible} to construct a non-reversible conical bicombing which has the midpoint property.

\begin{lemma}\label{lem:midpoint-property}
Let \(\sigma\colon D(X)\to X\) be defined as in Lemma~\ref{Lem:non-reversible}. Then the map $\tau\colon D(X)\to X$ defined by the assignment
\begin{equation*}
(p,q,t)\mapsto
\begin{cases}
\sigma\left(p, \frac{1}{2}\left(\sigma(p,q,\frac{1}{2})+\sigma(q,p,\frac{1}{2})\right), 2t\right), &\textrm{if } t\in [0,\frac{1}{2}], \\
\sigma\left(\frac{1}{2}\left(\sigma(p,q,\frac{1}{2})+\sigma(q,p,\frac{1}{2})\right), q, 2t-1\right), &\textrm{if } t\in [\frac{1}{2},1],
\end{cases}
\end{equation*}
is a non-reversible conical bicombing on $(X, \norm{\,\cdot\,}_\infty)$ that has the midpoint property. 
\end{lemma}

\begin{proof}
It is readily verified that $\tau$ is a conical  bicombing with the midpoint property. To see that $\tau$ is non-reversible, take for instance $p:=(-\frac{3}{2},\frac{1}{2})$, $q:=(0,\frac{1}{2})$ and observe that \[\tau(p,q,\tfrac{5}{12}) = (-\tfrac{7}{8}, \tfrac{1}{8}) \neq (-\tfrac{7}{8},\tfrac{1}{48}) = \tau(q,p,\tfrac{7}{12});\] compare Figure~\ref{Fig:midpoint}.
\end{proof}

To prove Proposition~\ref{Prop:reversible} we need the following midpoint construction:

\begin{lemma}
	Let $X$ be a complete metric space. If $\sigma\colon D(X)\to X$ is a conical  bicombing, then there is a midpoint map $m \colon X \times X \to X$ with the following properties: For all $x,y, \bar{x}, \bar{y} \in X$ we have
\begin{enumerate}[(i)]
	\item $m(x,y)=m(y,x)$,
	\item $d(x,m(x,y)) = d(y,m(x,y)) = \frac{1}{2}d(x,y)$,
	\item $d(m(x,y),m(\bar{x},\bar{y})) \leq \frac{1}{2}d(x,\bar{x}) + \frac{1}{2}d(y,\bar{y})$.
\end{enumerate}
\end{lemma}

\begin{proof}
	Let $x,y \in X$. Set $x_0 := x$, $y_0 := y$ and define recursively $x_{n+1} := \sigma(x_n,y_n,\frac{1}{2})$, $y_{n+1}:= \sigma(y_n,x_n,\frac{1}{2})$.
	We have
\begin{align*}
	d(x_{n+1},y_{n+1}) &= d(\sigma(x_n,y_n,\tfrac{1}{2}),y_{n+1}) \\
	&\leq \tfrac{1}{2} d(x_n, y_{n+1}) + \tfrac{1}{2} d(y_n, y_{n+1}) = \tfrac{1}{2} d(x_n,y_n),
\end{align*}
	and therefore $d(x_n,y_n) \leq \frac{1}{2^n} d(x,y)$, $d(x_n,x_{n-1}) \leq \frac{1}{2^n} d(x,y)$. Hence the sequences $(x_n)_{n\geq 0}$, $(y_n)_{n\geq 0}$ are Cauchy and converge to some common limit point $m(x,y)$. By the construction, we clearly have $(i)$. To prove $(ii)$ we claim that 
	\[
	d(x,x_n),\,d(x,y_n),\,d(y,x_n),\,d(y,y_n) \leq \frac{1}{2} d(x,y)
	\] 
for all $n \geq 1$. This follows by induction since \[d(x,x_{n+1}) \leq \frac{1}{2}d(x,x_n) + \frac{1}{2}d(x,y_n) \leq \frac{1}{2} d(x,y)\] and similar for all other distances.
	It remains to show $(iii)$. If we repeat the construction for $\bar{x},\bar{y} \in X$ we get some sequences $(\bar{x}_n)_{n\geq 0}$, $(\bar{y}_n)_{n\geq 0}$ with limit point $m(\bar{x},\bar{y})$. We now prove by induction that $d(x_n,\bar{x}_n), d(y_n,\bar{y}_n) \leq \frac{1}{2}d(x,\bar{x}) + \frac{1}{2}d(y,\bar{y})$ for all $n \geq 1$. Indeed, we have
	\begin{align*}
		d(x_{n+1},\bar{x}_{n+1}) &= d(\sigma(x_n,y_n,\tfrac{1}{2}),\sigma(\bar{x}_n,\bar{y}_n,\tfrac{1}{2})) \\
		&\leq \tfrac{1}{2}d(x_n, \bar{x}_n) + \tfrac{1}{2}d (y_n, \bar{y}_n) \leq \tfrac{1}{2}d(x,\bar{x}) + \tfrac{1}{2}d(y,\bar{y}),
	\end{align*}
	and similarly $d(y_{n+1},\bar{y}_{n+1}) \leq \tfrac{1}{2}d(x,\bar{x}) + \tfrac{1}{2}d(y,\bar{y})$. Hence, statement $(iii)$ follows by taking the limit \(n\to \infty\).  
\end{proof}

\begin{proof}[Proof of Proposition~\ref{Prop:reversible}.]
	We define a new bicombing $\tau \colon D(X) \to X$ by
	\begin{align}
		\tau (x,y,t) := m( \sigma(x,y,t), \sigma(y,x,1-t)).
	\end{align}
	For two points $x,y \in X$ this defines a geodesic from $x$ to $y$, since for $s,t \in [0,1]$ we have
	\begin{align*}
		d(\tau(x,y,t),\tau(x,y,s)) &= d(m(\sigma(x,y,t),\sigma(y,x,1-t)),m(\sigma(x,y,s),\sigma(y,x,1-s)) \\
		&\leq \tfrac{1}{2} d(\sigma(x,y,t),\sigma(x,y,s)) + \tfrac{1}{2} d(\sigma(y,x,1-t),\sigma(y,x,1-s)) \\
		&= |s-t| d(x,y).
	\end{align*}
	Moreover, the conical inequality holds, as we have
	\begin{align*}
		d(\tau(x,y,t),\tau(\bar{x},\bar{y},t)) &= d(m(\sigma(x,y,t),\sigma(y,x,1-t)),m(\sigma(\bar{x},\bar{y},t),\sigma(\bar{y},\bar{x},1-t))) \\
		&\leq \tfrac{1}{2} d(\sigma(x,y,t),\sigma(\bar{x},\bar{y},t)) + \tfrac{1}{2} d(\sigma(y,x,1-t),\sigma(\bar{y},\bar{x},1-t)) \\
		&\leq (1-t) d(x, \bar{x}) + t d(y,\bar{y}),
	\end{align*}
	for all $x$, $y$, $\bar{x}$, $\bar{y} \in X$ and all $t \in [0,1]$.
\end{proof}


\section{Local behavior of conical  bicombings }\label{Sec:uniqueness}


The goal of this section is to establish the following rigidity result. 

\begin{theorem}\label{Thm:linear}
Let $V$ be a normed vector space. Suppose that $A\subset V$ admits a conical  bicombing $\sigma\colon D(A)\to A$ and let $p, q\in A$. If there are extreme points $e_1, \ldots, e_n $ of $B_1$ and a tuple $(\lambda_1, \ldots, \lambda_n)\in [0,1]^n$ with $\sum_{k=1}^n \lambda_k=1$ such that 
\begin{align}
&\frac{p-q}{2}=\frac{\norm{p-q}}{2}\sum_{k=1}^n\lambda_k e_k  \textrm{ and } \label{Eq:linear1}\\
&\frac{p+q}{2}+\frac{\norm{p-q}}{2}\Big\{ \sum_{k=1}^{n} (-1)^{\epsilon_k}\lambda_k e_k : (\epsilon_1, \ldots, \epsilon_{n})\in \{ 0,1\}^{n}\Big\}\subset A, \label{Eq:linear2} 
\end{align}
then $\sigma(p, q, t)=(1-t)p+tq$ for all $t\in [0,1]$.  
\end{theorem}

Throughout this section we use the following notation:
\begin{align*}
B_{r}(p)&:=\{ q\in V : \, \norm{p-q}\leq r\},&  S_{r}(p)&:=\{ q\in V : \, \norm{p-q}=r\},
\end{align*}
and we abbreviate \(B_r:=B_r(0)\) and \(S_r:=S_r(0)\). Theorem~\ref{Thm:linearInterior} is a direct consequence of Theorem~\ref{Thm:linear}:

\begin{proof}[Proof of Theorem~\ref{Thm:linearInterior}.]
Let \(p,q\in  B_{r}(p_0)\). We have $\frac{p+q}{2} \in B_r(p_0)$ and $\frac{\|p-q\|}{2} \leq r$, so \[B_{\frac{\|p-q\|}{2}}(\frac{p+q}{2})\subset A.\] 
Then by Theorem~\ref{Thm:linear} and a straightforward limit argument, it follows that $\sigma(p,q,t)=(1-t)p+tq$ for all \(t\in [0,1]\), as desired. 
\end{proof}

We prove Theorem~\ref{Thm:linear} by induction on the number of extreme points. For this induction, we need some preparatory lemmas and definitions. For each $t\in [0,1]$ we set
\begin{equation*}
M^{(t)}(p,q):=\{ z\in V : \,\norm{z-p}=t\norm{p-q}, \,\norm{z-q}=(1-t)\norm{p-q} \} .
\end{equation*}
Clearly, \(\sigma(p,q,t)\in M^{(t)}(p,q)\) for every  bicombing \(\sigma\). Thus, if \(M^{(t)}(p,q)\) is a singleton, then \(\sigma(p,q,t)=\lambda(p,q,t)\). Here, $\lambda\colon D(V)\to V$  is the linear bicombing defined by \((p,q,t)\mapsto (1-t)p+tq \). Our first lemma gives a sufficient condition for the set \(M^{(t)}(p,q)\) to be a singleton. 

\begin{lemma}\label{Lem:extremepoint}
Let $V$ be a normed vector space and \(p\in V\). If \(p\) is an extreme point of $B_{r}$, for \(r=\norm{p}\), then $ M^{(t)}(p,-p)=\{(1-2t)p\}$ for all $t\in [0,1]$.
\end{lemma}

\begin{proof}
We abbreviate \(r:=\norm{p}\). By construction, 
\begin{equation*}
M^{(t)}(p,-p)=\left(S_{2tr}+p\right)\cap \left(S_{(1-t)2r}-p \right);
\end{equation*}
hence,
\begin{equation}\label{Eq:relationME}
\frac{1}{2t}\big(p-M^{(t)}(p,-p) \big)= S_{r}\cap \big(\frac{1}{t}p-\frac{1-t}{t}S_{r} \big),
\end{equation}
provided that $t\in (0,1]$. For each \(t\in (0,1]\) and \(p\in V\), we define
\begin{equation*}
E^{(t)}(p)\mapsto S_{r}\cap \big(\frac{1}{t}p-\frac{1-t}{t}S_{r} \big),  
\end{equation*} 
By the use of \eqref{Eq:relationME}, we find that $ M^{(t)}(p,-p)=\{(1-2t)p\}$ if and only if  \(E^{(t)}(p)=\{p\}\). Thus, to finish the proof we need to show that if $p$ is an extreme point of $B_{r}$, then $E^{(t)}(p)=\{p\}$ for all $t\in(0,1)$. We show the contraposition of this statement. Suppose there are $t\in(0,1)$ and $p' \in E^{(t)}(p)$ with $p'\neq p$. As $p'\in E^{(t)}(p)$, it follows  $p'\in S_{r}$ and there is $q\in S_{r}$ such that $p'= \frac{1}{t}p-\frac{1-t}{t}q$. We observe that $q \neq p$ and
\begin{equation*}
(1-t)q + tp' = (1-t)q+t\big(\frac{1}{t}p-\frac{1-t}{t}q\big)=p.
\end{equation*}
Hence, $p$ is not an extreme point of $B_{r}$, as desired. The lemma follows. 
\end{proof}

Lemma~\ref{Lem:extremepoint} will serve as base case for the induction in the proof of Theorem~\ref{Thm:linear}.
The following lemma is the key component of the inductive step. 

\begin{lemma}\label{Lem:linear}
Let $V$ be a normed vector space, and let $A\subset V$ be a subset admitting a conical  bicombing $\sigma$. Let $p\in A$ such that $-p\in A$. If there is $z\in V$ such that the points $2z-p$ and $p-2z$ are contained in $A$, $\sigma(p, p-2z, \cdot)=\lambda(p, p-2z, \cdot)$ and $\sigma(2z-p,-p, \cdot)=\lambda(2z-p,-p, \cdot)$, then 
\begin{equation*}
\sigma(p,-p,t)\in \big((1-2t)z+M^{(t)}(p-z, z-p)\big)
\end{equation*}
for all $t\in [0,1]$. 
\end{lemma}

\begin{proof}
Let $t\in [0,1]$. Then, using that \(\sigma\) is conical, we compute
\begin{align*}
\norm{\sigma(p,-p,t)-\lambda(p, p-2z, t) } &\leq 2t\norm{p-z} \\
\norm{\sigma(p,-p,t)-\lambda(2z-p,-p, t)} &\leq 2(1-t)\norm{p-z}. 
\end{align*}
Note that $\norm{\lambda(p, p-2z, t)-\lambda(2z-p,-p, t)}=2\norm{p-z}$. Therefore, it follows that 
\begin{equation*}
\sigma(p,-p,t)\in M^{(t)}\left(\lambda(p, p-2z, t), \lambda(2z-p,-p, t)\right).
\end{equation*} 
Clearly, $M^{(t)}(u+h,v+h)=h+M^{(t)}(u,v)$ for all $t\in [0,1]$ and all $u,v,h\in V$.
Consequently, we obtain 
\begin{equation*}
M^{(t)}\left(\lambda(p, p-2z, t), \lambda(2z-p,-p, t)\right)=(1-2t)z+M^{(t)}\left(p-z, z-p\right),
\end{equation*} 
as desired.
\end{proof}

\begin{proof}[Proof of Theorem~\ref{Thm:linear}]
It suffices to consider the case \(q=-p\). 
We proceed by induction on $n\geq 1$. If $n=1$, then Lemma~\ref{Lem:extremepoint} tells us that 
\begin{equation*}
\sigma\left(p, -p, t\right)=(1-2t)p
\end{equation*}
for all $t\in[0,1]$, as desired.
 
Suppose now $n>1$ and the statement holds for $n-1$. We may assume that $\lambda_1\in (0,1)$. 
We define $(\lambda^\prime_1, \ldots, \lambda_{n-1}^\prime):=\frac{1}{1-\lambda_1}(\lambda_2, \ldots, \lambda_n)$ and $(e_1^\prime, \ldots, e_{n-1}^\prime):=(e_2, \ldots, e_n)$. Note that  
\begin{equation}\label{Eq:decomposition}
\sum_{k=1}^n \lambda_k e_k=\lambda_1 e_1+(1-\lambda_1)\sum_{k=1}^{n-1} \lambda_k^\prime e^\prime_k.
\end{equation}
 Further, 
\begin{equation}\label{Eq:normed}
\Vert \sum_{k=1}^{n-1} \lambda_k^\prime e_k^\prime\Vert =1,
\end{equation}
as otherwise \eqref{Eq:decomposition}  implies 
\[\Vert \sum_{k=1}^n \lambda_k e_k\Vert < 1, \]
which is not possible due to \eqref{Eq:linear1}. We abbreviate $r:=\frac{\norm{p-q}}{2}$ and we set 
\begin{align*}
	z&:= r(1-\lambda_1)\sum_{k=1}^{n-1} \lambda_k^\prime e_k^\prime, 
	&
	p'&:=p, & q'&:=p'-2z.
\end{align*}
By the use of \eqref{Eq:normed}, we obtain
\begin{equation}\label{Eq:(p'-q')norm}
\frac{\norm{p'-q'}}{2}  =  r(1-\lambda_1).
\end{equation}
We have 
\begin{equation*}
\frac{p'+q'}{2} = p - z \overset{\eqref{Eq:linear1}}{=} r\sum_{k=1}^n \lambda_k e_k - r(1-\lambda_1)\sum_{k=1}^{n-1} \lambda_k^\prime e_k^\prime \overset{\eqref{Eq:decomposition}}{=} r \lambda_1 e_1
\end{equation*}
and therefore
\begin{align*}
	&\frac{p'+q'}{2} + \frac{\norm{p'-q'}}{2} \Big\{ \sum_{k=1}^{n-1} (-1)^{\epsilon_k}\lambda_k^\prime e_k^\prime : (\epsilon_1, \ldots, \epsilon_{n-1}) \in \{ 0,1\}^{n-1} \Big\} \\
&\quad \overset{\eqref{Eq:(p'-q')norm}}{=} r \Big\{ \lambda_1e_1+\sum_{k=2}^{n} (-1)^{\epsilon_k}\lambda_k e_k : (\epsilon_2 \ldots, \epsilon_{n})\in \{ 0,1\}^{n-1}\Big\}\overset{\eqref{Eq:linear2}}{\subset} A. 
\end{align*}
Thus, we can apply the induction hypothesis to \(p',q'\in A\) and deduce  
\begin{align*}
\sigma(p', p'-2z, \,\cdot \,) = \lambda(p',p'-2z,\,\cdot\,).
\end{align*}
Similarly, we find that 
\begin{equation*}
\sigma(2z-p',-p', \,\cdot\,) = \lambda(2z-p',-p',\,\cdot\,).
\end{equation*}
Now, by the use of Lemma~\ref{Lem:linear} it follows 
\begin{equation*}
\sigma(p',-p',t )\in \big( (1-2t)z+M^{(t)}\left(p'-z, z-p'\right)\big)
\end{equation*}
for all $t\in [0,1]$; consequently, we get 
\begin{equation*}
\sigma(p',-p',t )=(1-2t)p',
\end{equation*}
since $p'-z=r\lambda_1e_1$ is an extreme point of $B_{r\lambda_1}$ and thus we can use Lemma~\ref{Lem:extremepoint} to deduce  $M^{(t)}(p'-z, z-p')=\{(1-2t)(p'-z)\}$. 
Hence, we have
\begin{align*}
 \sigma(p,q,t) = \sigma(p',-p',t ) = (1-2t)p,
\end{align*}
as desired.
\end{proof}
We conclude this section with an example of a closed convex set admitting two distinct consistent conical bicombings.

\begin{example}\label{Ex:twoDistinct}
The following construction is inspired by a similar construction due to
Schechtman. Let \(A\) denote the set of all continuous, strictly increasing functions \(f\colon [0,1]\to [0,1]\) such that \(f(0)=0\) and \(f(1)=1\).
We claim that \((A, \norm{\cdot}_1)\) admits two distinct consistent conical  bicombings. Clearly, as \(A\) is convex, \(\lambda\colon D(A)\to A\) defined by \((f,g,t)\mapsto (1-t)f+tg\) is a consistent conical  bicombing on \((A, \norm{\cdot}_1)\). Graphically, this corresponds to a vertical interpolation of the functions \(f\), \(g\in A\).

Next, we consider the bicombing obtained by interpolating functions \(f\), \(g\in A\) horizontally. To this end, let \(\varphi\colon A\to A\) be defined by \(f\mapsto f^{-1}\). It follows that \(\varphi\) is an isometry of \((A, \norm{\cdot}_1)\). Indeed, this is a simple consequence of the identity
\begin{equation*}
\norm{f-g}_1=\textrm{vol}_2\left(\big\{ (x,y)\in [0,1]^2 : \min\{f(x),g(x)\} \leq y \leq \max\{f(x),g(x)\}\big\}\right)
\end{equation*} 
which holds true for all \(f,g\in A\) and where \(\textrm{vol}_2\) denotes the two-dimensional Lebesgue measure. Let \(\tau\colon D(A)\to A\) be given by the assignment 
\[
(f,g, t)\mapsto \varphi\big( (1-t)\varphi(f)+t\varphi(g)\big).
\]
As \(\varphi\) is an isometry, it follows that \(\tau\) is a consistent conical bicombing. Furthermore, if \(f(x):=\sqrt{x}\) and \(g(x):=x\), then \(\tau( f, g, t)\colon [0,1]\to [0,1]\) is given by 
\begin{equation*}
x\mapsto \frac{-t+\sqrt{4(1-t)x+t^2}}{2(1-t)}
\end{equation*}
for all \(t\in [0,1]\), which is distinct from \(\lambda(f,g,t)=(1-t)f+tg\) for all \(t\in (0,1)\). Hence, \((A, \norm{\cdot}_1)\) admits two distinct consistent conical bicombings, as claimed. Now, let \(B\) denote the closure of \(A\subset L^1([0,1])\). Note that \(\lambda\) and \(\tau\) naturally extend to consistent conical bicombings on \(B\). Hence, we have constructed a closed convex subset of a Banach space that admits two distinct consistent conical bicombings. It is readily verified that \(B\) has empty interior.
\end{example}


\section{Proof of Theorem~\ref{Thm:uniqueness}}\label{sec:ProofofTheorem1.4}


Before we start with the proof of Theorem~\ref{Thm:uniqueness}, let us recall some notions from \cite{miesch2015cartan}. Let \(X\) be a metric space, \(p\in X\) and \(r>0\). We set 
\[
U_{r}(p):=\{q\in X : d(p,q) < r\}.
\]
Let \(U\subset D(X)\) be a subset. We say that \(\sigma\colon U\to X\) is a \textit{convex local  bicombing} if for every \(p\in X\) there is a real number \(r_p>0\) such that
\begin{equation*}
U=\bigcup_{p\in X} D(U_{r_p}(p)). 
\end{equation*}
and if the restriction \(\sigma|_{D(U_{r_p}(p))}\colon D(U_{r_p}(p))\to X\) is a consistent conical bicombing. Furthermore, we say that a geodesic \(c\colon [0,1]\to X\) is \textit{consistent} with the convex local  bicombing \( \sigma \) if for each choice of \(0 \leq s_1 \leq s_2\leq 1\) with  \((c(s_1), c(s_2))\in U_{r_p}(p)\times U_{r_p}(p)\) for some \(p\in X\), it holds 
\begin{equation*}
c((1-t)s_1+ts_2)=\sigma(c(s_1), c(s_2), t)
\end{equation*}
for all \(t\in [0,1]\). Consistent geodesics are uniquely determined by the local  bicombing, compare \cite[Theorem~1.1]{miesch2015cartan} and the proof thereof:

\begin{theorem}\label{Thm:Cartan-Hadamard}
	Let $X$ be a complete, simply-connected metric space with a convex local  bicombing $\sigma$. If we equip \(X\) with the length metric, then for every two points $p,q \in X$ there is a unique geodesic from $p$ to $q$ which is consistent with \(\sigma\) and the collection of all such geodesics is a convex  bicombing.
\end{theorem}

With Theorem~\ref{Thm:Cartan-Hadamard} on hand it is possible to derive Theorem~\ref{Thm:uniqueness} by the use of Theorem~\ref{Thm:linearInterior}.

\begin{proof}[Proof of Theorem~\ref{Thm:uniqueness}]
Let $\intt{\hspace{-0.1em}(C)}$ denote the interior of $C$ and let $p$, $q\in \intt{\hspace{-0.1em}(C)}$. We abbreviate 
\begin{equation*}
[p,q]:=\big\{ (1-t)p+tq : t\in [0,1]\big\}.
\end{equation*}
As $\intt{\hspace{-0.1em}(C)}$ is convex, we have $[p,q]\subset \intt{(C)}$. For each $z\in C$ we set
\begin{equation*}
r_z:=
\begin{cases}
\min\{ \norm{z-w} : w\in [p,q]\} & \textrm{ if } z\in C\setminus{\intt{\hspace{-0.1em}(C)}}\\
\frac{1}{2}\inf\left\{\norm{z-w} : w\in C\setminus{\intt{\hspace{-0.1em}(C)}}\right\}& \textrm{ if } z\in \intt{\hspace{-0.1em}(C)}. 
\end{cases}
\end{equation*}
Note that \(r_z >0\) for all \(z\in C\) and \(U_{r_z}(z)\cap [p,q]=\varnothing\) if \(z\in C\setminus{\intt{\hspace{-0.1em}(C)}}\). 
Further, for every \(z\in \intt{\hspace{-0.1em}(C)}\) it follows that $B_{2r_z}(z)\subset C$; thus, we may invoke Theorem~\ref{Thm:linearInterior} to deduce that if \(z\in \intt{\hspace{-0.1em}(C)}\), then $\sigma_{z_1z_2}(t)=(1-t)z_1+tz_2$ for all $z_1, z_2\in B_{r_z}(z)$ and all $t\in [0,1]$. We define 
\begin{equation*}
U:=\bigcup_{z\in C} D(U_{r_z}(z)).
\end{equation*}
Note that $\sigma^{\textrm{loc}}:=\sigma|_{U}$ defines a convex local bicombing on $C$. The geodesic $\sigma_{pq}(\cdot)$ and the linear geodesic from $p$ to $q$ are both consistent with $\sigma^{\textrm{loc}}$. Hence, by Theorem~\ref{Thm:Cartan-Hadamard}, we conclude that $\sigma_{pq}(\cdot)$ must be equal to the linear geodesic from $p$ to $q$, that is, we have $\sigma_{pq}(t)=(1-t)p+tq$ for all $t\in [0,1]$. 

Now, suppose \(p,q \in C\). As \(C\) is convex, it is well-known that $C = \overline{\intt{(C)}}$, see \cite[Lemma 5.28]{aliprantis2006infinite}. Let \((p_n), (q_n)\subset \intt{(C)}\) be two sequences such that \(p_n\to p\) and \(q_n\to q\) as \(n\to \infty\). Since \(\sigma\) is a conical  bicombing, it is readily verified that \(\sigma_{p_n q_n}(\cdot)\to \sigma_{pq}(\cdot)\). Hence, \(\sigma_{pq}(\cdot)\) is equal to the linear geodesic from $p$ to $q$. 
\end{proof}

\appendix
\section{Proofs of Propositions~\ref{Prop:convex} and \ref{Prop:convex'}}


For the sake of completeness, we add here the remaining, quite technical details in the proofs of Propositions~\ref{Prop:convex} and \ref{Prop:convex'}, which were stated in Section~\ref{Sec:convex_non-consistent}.

\begin{proof}[Proof of Proposition~\ref{Prop:convex}.]
Recall that \(X\) is equipped with the metric \(d\) induced by the norm
\[
	\| (x,y) \| := \max \big\{ |x| , \tfrac{\sqrt{2}}{2} \| (x,y) \|_2 \big\}.
\]
Clearly,  $\| (x,y) \| = |x|$ if and only if $|y| \leq |x|$. Thus, if  
$d(\sigma_{pq}^\delta(t),\sigma_{p'q'}^\delta(t)) = |x_{pq}(t)-x_{p'q'}(t)|$, then 
\begin{align*}
	2d(\sigma_{pq}^\delta(t),\sigma_{p'q'}^\delta(t))
	&\leq d(\sigma_{pq}^\delta(t-\tau),\sigma_{p'q'}^\delta(t-\tau))+d(\sigma_{pq}^\delta(t+\tau),\sigma_{p'q'}^\delta(t+\tau)).
\end{align*}
Therefore, to show that \(\sigma^\delta\) is a convex bicombing, it remains to prove that
\[
	2 \|\sigma_{pq}^\delta(t)-\sigma_{p'q'}^\delta(t)\|_2 \leq \|\sigma_{pq}^\delta(t-\tau)-\sigma_{p'q'}^\delta(t-\tau)\|_2+\|\sigma_{pq}^\delta(t+\tau)-\sigma_{p'q'}^\delta(t+\tau)\|_2
\]
for $\tau > 0$ sufficiently small, whenever $d(\sigma_{pq}^\delta(t),\sigma_{p'q'}^\delta(t)) = \tfrac{\sqrt{2}}{2} \|\sigma_{pq}^\delta(t)-\sigma_{p'q'}^\delta(t)\|_2$, that is, 
\[|x_{pq}(t)-x_{p'q'}(t)| \leq |y_{pq}(t)-y_{p'q'}(t)|.
\] 
For $(x, 0)\in X_-\cup X_+$ and $(x',y') \in X$, it holds $d((x,0),(x',y'))=|x-x'|$ and therefore 
\[
d(\sigma_{pq}^\delta(t),\sigma_{p'q'}^\delta(t)) = |x_{pq}(t)-x_{p'q'}(t)|
\] 
provided \(x_{pq}(t) \notin (-1,1)\). As a result, we only need to consider points that satisfy $x_{pq}(t),x_{p'q'}(t) \in (-1,1)$. From now on this will be a standing assumption.

First, if $\sigma_{pq}^\delta$, $\sigma_{p'q'}^\delta$ are (piece-wise) linear, then locally they are linear geodesics in a normed space and hence $d(\sigma_{pq}(t),\sigma_{p'q'}(t)) = \|\sigma_{pq}(t)-\sigma_{p'q'}(t)\|$ is locally convex, thus convex. Let us now assume that $\sigma_{pq}^\delta$ is not linear, that is, $p\in X_-$, $q\in X_+$, $l := d(p,q) \geq 4$. In the following, we look at the different options for $\sigma_{p'q'}^\delta$ separately. But before doing so, let us first fix some notation. We define 
\[
p_0 := \sigma_{pq}(t), \quad p_\pm := \sigma_{pq}(t\pm \tau), \quad p_\ast=(x_\ast,y_\ast) \quad (\ast \in \{0,+,-\}),
\] 
and accordingly for $\sigma_{p'q'}^\delta$. Further, we set $D := \delta (l-4)$ and $\epsilon := \tau l$. We then get $y_0=D(1-x_0^2)$,
\[
x_{\pm} = x_0 \pm \epsilon \quad \text{ and }  \quad y_{\pm} = D(1-(x_0 \pm \epsilon)^2).
\]
We proceed by a case distinction involving three cases. In each case, we also tacitly assume that $x_0$, $x_0' \in (-1,1)$ and $|x_0-x_0'| \leq |y_0-y_0'|$. \\
\\
\noindent\textbf{Case 1.} We have $l' := d(p',q') \in [4,l)$ with $p' \in X_\mp$ and $q' \in X_\pm$. We set $\lambda = \frac{l'}{l}$ and $\epsilon'=\lambda \epsilon$. As above, $y_0'=D'(1-x_0'^2), \,\, x'_{\pm} = x_0' \pm \epsilon',$ and $ y_{\pm}' = D'(1-(x_0' \pm \epsilon')^2)$. We claim that
\begin{align*}
	2 \|p_0-p'_0\|_2 &\leq \|p_--p_-'\|_2 + \|p_+-p_+'\|_2
\end{align*}
for $\epsilon>0$ (i.e. $\tau > 0$) sufficiently small. First note that
\[
	\|p_--p_-'\|_2^2 = \|p_0-p'_0\|_2^2 - 2(x_0-x_0')(1-\lambda)\epsilon + (1-\lambda)^2 \epsilon^2 + 2(y_0-y_0')a\epsilon + a^2 \epsilon^2, 
	\]
and
\[
\|p_+-p_+'\|_2^2 = \|p_0-p'_0\|_2^2 + 2(x_0-x_0')(1-\lambda)\epsilon +(1-\lambda)^2 \epsilon^2 + 2(y_0-y_0')b\epsilon + b^2 \epsilon^2,
\]
for 
\begin{align*}
	a &:= 2(x_0D-\lambda x_0'D') - (D-\lambda^2 D')\epsilon, \\[0.1em]
	b &:= -2(x_0D-\lambda x_0'D') - (D-\lambda^2 D')\epsilon,
\end{align*}
with 
\[a+b = -2(D-\lambda^2D')\epsilon, \quad a-b = 4(x_0D-\lambda x_0'D')\]
and either $ab=(D-\lambda^2 D')^2\epsilon^2$ or $ab < 0$ for $\epsilon$ sufficiently small. In the following, we assume $ab<0$. The proof of the other case is similar. We compute
\begin{align*}
	&\|p_--p_-'\|_2^2 \cdot\|p_+-p_+'\|_2^2 = \|p_0-p'_0\|_2^4 \\[0.5em]
	&+ \epsilon^2 \cdot \Big[4ab(y_0-y_0')^2 - 4(x_0-x_0')^2(1-\lambda)^2 + 4(x_0-x_0')(y_0-y_0')(1-\lambda)(a-b) \\[0.5em]
	&+ \big(2(1-\lambda)^2  +a^2 + b^2 - 4(y_0-y_0')(D-\lambda^2D') \big) \cdot \|p_0-p'_0\|_2^2\Big]  + \O(\epsilon^3) 
\end{align*}
	and with $\sqrt{u+t} = \sqrt{u} + \frac{t}{2\sqrt{u}} + \O(t^2)$ and $u= \|p_0-p_0'\|_2^4$,  it follows
\begin{align*}
	2 &\sqrt{\|p_--p_-'\|_2^2 \cdot\|p_+-p_+'\|_2^2} \geq 2 \|p_0-p'_0\|_2^2 \\[0.5em]
	&+ \epsilon^2 \cdot \Big[2(1-\lambda)^2  +a^2 + b^2 + 4ab - 4(y_0-y_0')(D-\lambda^2D')\\[0.5em]
	&\quad + \frac{4(x_0-x_0')(y_0-y_0')(1-\lambda)(a-b)}{(x_0-x_0')^2 + (y_0-y_0')^2}- \frac{4(x_0-x_0')^2(1-\lambda)^2}{(x_0-x_0')^2 + (y_0-y_0')^2} \Big]+ \O(\epsilon^3),
\end{align*}
where we used that \(ab <0\). We get
\begin{align*}
	\big(\|&p_--p_-'\|_2 + \|p_+-p_+'\|_2\big)^2 \geq  4\|p_0-p'_0\|_2^2 \\[0.5em]
	&+ \epsilon^2 \cdot \Big[4(1-\lambda)^2 + 2(a+b)^2- 8(y_0-y_0')(D-\lambda^2D')\\[0.5em]
	&+ \frac{4(x_0-x_0')(y_0-y_0')(1-\lambda)(a-b)}{(x_0-x_0')^2 + (y_0-y_0')^2}- \frac{4(x_0-x_0')^2(1-\lambda)^2}{(x_0-x_0')^2 + (y_0-y_0')^2} \Big] + \O(\epsilon^3).
\end{align*}
We set
\begin{align*}
	C &:=  4(1-\lambda)^2 - 8(y_0-y_0')(D-\lambda^2D') \\[0.5em]
	&\quad + \frac{16(x_0-x_0')(y_0-y_0')(1-\lambda)(x_0D-\lambda x_0'D')}{(x_0-x_0')^2 + (y_0-y_0')^2}- \frac{4(x_0-x_0')^2(1-\lambda)^2}{(x_0-x_0')^2 + (y_0-y_0')^2}.
\end{align*}
By the above,
\[
\big(\|p_--p_-'\|_2 + \|p_+-p_+'\|_2\big)^2\geq 4\|p_0-p'_0\|_2^2 + C \epsilon^2 + \O(\epsilon^3)\geq  4\|p_0-p'_0\|_2^2
\]
for $\epsilon>0$ sufficiently small, provided that \(C>0\). In the following we show that \(C>0\). Assuming $y_0 > y_0'$, we obtain
\begin{align*}
	y_0-y_0'&= (D-D')(1-x_0^2) + D'(x_0'^2-x_0^2) \\[0.5em]
	&\leq \delta (l-l') + \delta (l'-4) (x_0'+x_0)(x_0'-x_0) \leq \delta (l-l') + 4\delta (y_0-y_0'),
\end{align*}
where we used that $D := \delta (l-4)$, our standing assumption that $x_{pq}(t)$, $x_{p'q'}(t) \in (-1,1)$ and also \(4 \leq l' < l \leq 6\). As a result, 
\begin{align*}
	|y_0-y_0'| \leq \frac{\delta}{1-4\delta}(l-l').
\end{align*}
	Moreover, we have 
\begin{align*}
	|D-\lambda^2D'|l^2 = \delta (l-l')(l^2 + ll' + l'^2 - 4(l+l')) \leq 60 \delta (l-l')
\end{align*}
and
\begin{align*}
	|x_0D-\lambda x_0'D'|l &\leq |x_0|(D-\lambda D')l + |x_0-x_0'|D'l' \\[0.5em]
	&\leq \delta (l-l')(l+l'-4) + 12 \delta |y_0-y_0'| \leq (8 \delta + \tfrac{12 \delta^2}{1-4\delta})(l-l').
\end{align*}
	Hence, using that
\begin{align*}
	C l^2 \cdot\|p_0-p_0'\|_2^2 &= \, 4(l-l')^2(y_0-y_0')^2 - 8(y_0-y_0')(D-\lambda^2D')l^2\left((x_0-x_0')^2 + (y_0-y_0')^2\right) \\[0.5em]
	&\quad + 16(x_0-x_0')(y_0-y_0')(l-l')(x_0D-\lambda x_0'D')l.
\end{align*}
we finally get 
\begin{align*}
C l^2 \cdot\|p_0-p_0'\|_2^2 &\geq \big(4- \tfrac{960\delta^2}{1-4\delta} - 128 \delta - \tfrac{192 \delta^2}{1-4\delta}\big) (l-l')^2(y_0-y_0')^2 >0
\end{align*}
for all $\delta <  \frac{1}{40}$. This is in particular true for all \(\delta \leq \frac{1}{64}\).\\
\\
\noindent\textbf{Case 2.} We have that $\sigma_{p'q'}$ is piecewise linear with $p' \notin X_0$ or $q' \notin X_0$. Notice that \(l'=d(p', q')\) satisfies  $l' \in [0,4]$ and also \(l'=\abs{p_x'-q_x'}\). Let $m$ be the slope of $\sigma_{p'q'}$ at $p_0'$. If $p' \in X_-$ and $q' \in X_0$, then
\begin{align*}
	m = \frac{q'_y}{q'_x+1} \leq  \frac{\frac{1}{32}(1-{q'_x}^2)}{1+q_x} = \tfrac{1}{32}(1-q'_x) \leq \tfrac{1}{32}(4-l'),
\end{align*}
and so \(\abs{m} \leq \tfrac{1}{32}(l-l')\). The same estimate also holds in the other cases.  We set $\lambda = \frac{l'}{l}$ and $\epsilon' = \lambda\epsilon$. Observe that $\epsilon' = \tau l'$, and also $x_{\pm}' = x_0' \pm \epsilon'$ and $y_{\pm}' = y_0' \pm m\epsilon'$. We can proceed as in Case 1. with
\begin{align*}
	a &= \lambda m + 2Dx_0 - D\epsilon, \\
	b &= -\lambda m	- 2Dx_0 - D\epsilon,
\end{align*}
and obtain the constant
\begin{align*}
	C &= 4(1-\lambda)^2 - 8(y_0-y_0')D \\[0.5em]
	&+ \frac{8(x_0-x_0')(y_0-y_0')(1-\lambda)(\lambda m +2Dx_0)}{(x_0-x_0')^2 + (y_0-y_0')^2}
	- \frac{4(x_0-x_0')^2(1-\lambda)^2}{(x_0-x_0')^2 + (y_0-y_0')^2}.
\end{align*}
Clearly, 
\begin{align*}
	C l^2 \cdot \|p_0-p_0'\|_2^2 &=  4(l-l')^2(y_0-y_0')^2 - 8(y_0-y_0')Dl^2 \left((x_0-x_0')^2 + (y_0-y_0')^2\right) \\[0.5em]
	&\quad + 8(x_0-x_0')(y_0-y_0')(l-l')(ml' + 2Dx_0l). 
\end{align*}	
Notice that
\begin{align*}
	D = \delta (l-4) \leq \delta(l-l') \quad \text{ and } \quad  y_0-y_0' &\leq D(1-x_0^2) \leq \delta(l-l'),
\end{align*}
where we used that \(y_0'\geq 0\) and \(x_0\in (-1,1)\). Thus, as \(\abs{m}\leq \tfrac{1}{32}(l-l')\), it follows that
\[
C l^2 \cdot \|p_0-p_0'\|_2^2\geq  \left(4- 576 \delta^2 - 1 - 96 \delta \right) (l-l')^2(y_0-y_0')^2 >0 
\]
for all $\delta <  0.026$. \\
\\
\noindent\textbf{Case 3.} We have that $\sigma_{p'q'}$ is linear with $p'$, $q' \in X_0$. Let $m$ again denote the slope of $\sigma_{p'q'}$. We distinguish two subcases. First, if $|m| \leq 1$, then it follows that \(l'=\abs{p_x'-q_x'}\) and so we find that 
\[
\abs{m l'}=\frac{\abs{p_y'-q_y'}}{\abs{p_x'-q_x'}}\, l'=\abs{p_y'-q_y'}\leq \tfrac{1}{32}\leq \tfrac{1}{64}(l-l'),
\]
where we used that \(l\in [4, 6]\) and \(l'\in [0,2]\).
For $\lambda = \frac{l'}{l}$ and $\epsilon' = \lambda\epsilon$, we get $x_{\pm}' = x_0' \pm \epsilon'$ and also $y_{\pm}' = y_0' \pm m\epsilon'$. Hence, by exactly the same reasoning as in Case 2., we obtain the constant
\begin{align*}
	C &=  4(1-\lambda)^2 - 8(y_0-y_0')D \\
	&\quad + \frac{8(x_0-x_0')(y_0-y_0')(1-\lambda)(\lambda m +2Dx_0)}{(x_0-x_0')^2 + (y_0-y_0')^2}- \frac{4(x_0-x_0')^2(1-\lambda)^2}{(x_0-x_0')^2 + (y_0-y_0')^2}.
\end{align*} 
We compute
\begin{align*}
	C l^2 \cdot \|p_0-p_0'\|_2^2 &= 4(l-l')^2(y_0-y_0')^2 - 8(y_0-y_0')Dl^2 \left((x_0-x_0')^2 + (y_0-y_0')^2\right) \\[0.5em]
	&\quad + 8(x_0-x_0')(y_0-y_0')(l-l')(ml' + 2Dx_0l), 
\end{align*}
and so it follows that
\[
	C l^2 \|p_0-p_0'\|_2^2\geq  \big(4- 576 \delta^2 - \tfrac{1}{8}- 96 \delta \big) (l-l')^2(y_0-y_0')^2 >0
\]	
for $\delta <  0.033$. We now treat the second subcase and assume that $|m|>1$.  In this case, we have 
\begin{align*}
	l' &= \tfrac{\sqrt{2}}{2} \sqrt{(q'_x-p'_x)^2+(q'_y-p'_y)^2}\leq |q'_y-p'_y| \leq \tfrac{1}{32}.
\end{align*}
Furthermore, let $\epsilon_x'= x'_+-x'_0$ and $\epsilon_y'=y'_+-y'_0$. Then $\epsilon_y' = m \epsilon_x'$ and
\begin{align*}	
	2(\tau l')^2 &= {\epsilon_x'}^2+(m\epsilon_x')^2 = (1+m^2){\epsilon_x'}^2.
\end{align*}
Thus we get $\epsilon_x' = \lambda_x \epsilon$ for $\lambda_x :=\lambda \cdot\frac{\sqrt{2}}{\sqrt{1+m^2}}$, and $\epsilon_y' = \lambda_y \epsilon$ for 
\[
\lambda_y := m \lambda_x = \lambda \cdot \frac{\sqrt{2}m}{\sqrt{1+m^2}}.
\]
Notice that $x_{\pm}' = x_0' \pm \epsilon'_x$ and $y_{\pm}' = y_0' \pm \epsilon'_y$. We proceed again as before and obtain
\begin{align*}
	C &= 4(1-\lambda_x)^2 - 8(y_0-y_0')D \\[0.5em]
	&\quad + \frac{8(x_0-x_0')(y_0-y_0')(1-\lambda_x)(\lambda_y +2Dx_0)}{(x_0-x_0')^2 + (y_0-y_0')^2}- \frac{4(x_0-x_0')^2(1-\lambda_x)^2}{(x_0-x_0')^2 + (y_0-y_0')^2}.
\end{align*}
Notice that \(\lambda_x\leq \frac{1}{128}\leq 1-\frac{\sqrt{2}}{2}\), and so
\[
\frac{\sqrt{2}}{2} \leq 1-\lambda_x,
\]
which implies that
\[
\lambda_y = \frac{l'}{l}\cdot \sqrt{\frac{2}{1+\frac{1}{m^2}}} \leq \frac{1}{64}(1-\lambda_x)
\]
as well as 
\begin{align*}
	D \leq \delta(l-l') \leq 6 \delta (1-\lambda_x).
\end{align*}
Since
\begin{align*}
	C \cdot\|p_0-p_0'\|_2^2 &=  4(1-\lambda_x)^2(y_0-y_0')^2 - 8(y_0-y_0')D \left((x_0-x_0')^2 + (y_0-y_0')^2\right) \\[0.5em]
	&\quad + 8(x_0-x_0')(y_0-y_0')(1-\lambda_x)(\lambda_y + 2Dx_0) 
\end{align*}	
we find by the use of \( y_0-y_0' \leq D\) that
\[
	C\cdot\|p_0-p_0'\|_2^2 \geq  \big(4- 576 \delta^2 - \tfrac{1}{64} - 96 \delta\big) (1-\lambda_x)^2(y_0-y_0')^2 >0
\]
for all $\delta <  0.034$. Observe that as $m \to \infty$, we get $\lambda_x=0$ and $\lambda_y =\sqrt{2} \cdot (l'/l)$, and the same estimates hold. We have considered all cases and so the proposition follows. 
\end{proof}	

\begin{proof}[Proof of Proposition~\ref{Prop:convex'}.]
As observed above, $\tilde{\sigma}^\delta$ is non-consistent and non-reversible. To show convexity, the same arguments as in the proof of Proposition~\ref{Prop:convex} apply. The only new case is $p' \in X_+$ and $q' \in X_-$. In this case, it holds that \(\tilde{\sigma}_{p'q'}^\delta(t) = (x_{p'q'}(t),0)\). With the notions from above with $x_\pm' = x_0' \mp \epsilon'$ for $\epsilon' = \tau l'$ and $\lambda = \frac{l'}{l}$, we obtain the constant
\begin{align*}
	C &=  4(1+\lambda)^2 - 8y_0D \\[0.5em]
	&\quad + \frac{16(x_0-x_0')y_0(1+\lambda)x_0D}{(x_0-x_0')^2 + y_0^2}- \frac{4(x_0-x_0')^2(1+\lambda)^2}{(x_0-x_0')^2 + y_0^2},
\end{align*}
and with the inequalities $D = \delta (l-4) \leq 2 \delta$ and $|y_0| \leq \frac{1}{32}$, we find that
\begin{align*}
	C \cdot\|p_0-p_0'\|_2^2 &=  4(1+\lambda)^2y_0^2 - 8y_0D\left((x_0-x_0')^2 + y_0^2\right) \\[0.5em]
	& + 16(x_0-x_0')y_0(1+\lambda)x_0D. 
\end{align*}
Hence, using our standing assumption that \(\abs{x_0-x_0'}\leq \abs{y_0-y_0'}\), we conclude 
\[
	C \cdot \|p_0-p_0'\|_2^2 \geq  \left(4 -\delta- 32 \delta\right)(1+\lambda)^2y_0^2 >0,
\]
for all $\delta < \frac{1}{64}$, which completes the proof.
\end{proof}


\bibliographystyle{plain}
\bibliography{refs}
\noindent
\textsc{\small{Mathematik Departement, ETH Zürich, Rämistrasse 101, 8092 Zürich, Schweiz}}\\
\textit{E-mail address:}{\textsf{ giuliano.basso@math.ethz.ch}}\\
\textit{E-mail address:}{\textsf{ benjamin.miesch@math.ethz.ch}}


\end{document}